\documentclass{amsart}
\usepackage{mathrsfs}
\usepackage{cases}
\usepackage{amssymb,amscd,amsthm}
\usepackage{latexsym}
\usepackage{hyperref}
\usepackage{color}
\date{\today}

\newcommand{\Z}{{\mathbb Z}}
\newcommand{\R}{{\mathbb R}}

\newcommand{\p}{\|}
\newcommand{\rarrow}{\rightarrow}

\newtheorem{thm}{Theorem} [section]
\newtheorem{remark}[thm]{Remark}
\newtheorem{lemma}[thm]{Lemma}
\newtheorem{prop}[thm]{Proposition}

\newtheorem{defn}[thm]{Definition}

\newcommand{\be}{\begin{equation}}
\newcommand{\ee}{\end{equation}}

\newcommand{\eps}{\varepsilon}

\newcommand{\sig}{\sigma}

\newcommand{\om}{\omega}

\begin{document}

\title[Limit-Periodic Schr\"odinger Operators]{Limit-Periodic Schr\"odinger
Operators With Uniformly Localized Eigenfunctions}

\author[D.\ Damanik]{David Damanik}

\address{Department of Mathematics, Rice University, Houston, TX~77005, USA}

\email {\href{mailto:damanik@rice.edu}{damanik@rice.edu}}

\urladdr {\href{http://www.ruf.rice.edu/~dtd3}{www.ruf.rice.edu/$\sim$dtd3}}

\author[Z.\ Gan]{Zheng Gan}

\address{Department of Mathematics, Rice University, Houston, TX~77005, USA}

\email{\href{mailto:zheng.gan@rice.edu}{zheng.gan@rice.edu}}

\urladdr{\href{http://math.rice.edu/~zg2/}{math.rice.edu/$\sim$zg2}}

\thanks{D.\ D.\ and Z.\ G.\ were supported in part by NSF grant
DMS--0800100.}


\begin{abstract}
We exhibit limit-periodic Schr\"odinger operators that are uniformly
localized in the strongest sense possible. That is, for these
operators there are uniform exponential decay rates such that every
element of the hull has a complete set of eigenvectors that decay
exponentially off their centers of localization at least as fast as
prescribed by the uniform decay rate. Consequently, these operators
exhibit uniform dynamical localization.
\end{abstract}

\maketitle

\section{Introduction}

This paper is a part of a sequence of papers exploring the spectral properties of discrete one-dimensional limit-periodic Schr\"odinger operators; see \cite{dg1,dg2} for the earlier papers in this sequence. The overarching goal is to obtain a spectral picture that is as complete as possible, that is, we explore which spectral phenomena can occur in this class of operators and how often they do so. The present paper is devoted to cases that display a strong form of localization.

Localization is a topic that has been explored in the context of Schr\"odinger operators to a great extent. By now several mechanisms are known that lead to localization, at least in suitable energy regions. The most important one is randomness or, more generally, weak correlations. This aspect goes back to the seminal paper \cite{and} of Anderson. Another important mechanism is strong coupling and, related to this, positive Lyapunov exponents. The latter approach can be used to prove localization for strongly correlated potentials.

On the other hand, localization does not occur for periodic potentials. Limit-periodic potentials are closest to periodic potentials (at least among the stationary ones) and hence for them, one would expect either the absence of localization or a difficult localization proof in the rare cases where it holds. Indeed, most of the work on limit-periodic potentials up to this point has focused on establishing continuous spectral type. There are two notable exceptions. The first is a paper by Chulaevsky and Molchanov, \cite{mc}, which unfortunately does not contain a proof of the theorem on the presence of pure point spectrum for some continuum one-dimensional limit-periodic Schr\"odinger operators stated there. Moreover, their examples have zero Lyapunov exponent and hence are not localized in the standard sense. The other relevant paper is P\"oschel's work \cite{p}, where he proves a general theorem that provides a sufficient condition for uniform localization along with two examples showing that the general result is applicable to limit-periodic potentials. Incidentally, the Chulaevsky-Molchanov paper uses features of randomness while P\"oschel's paper uses strong coupling.

Our goal in this paper is to explore the applicability of P\"oschel's general theorem in the realm of limit-periodic potentials. The setting we will use here was initially suggested by Avila in \cite{a} and has been consistently pursued in our previous papers \cite{dg1,dg2}. The idea is to regard limit-periodic potentials as dynamically defined potentials, where the base dynamics is a minimal translation of a Cantor group and the sampling function is continuous. By separating base dynamics and sampling function in this way, it becomes easy and natural to answer questions of the type \emph{how often does phenomenon X occur}? Here we will show that P\"oschel's results can be applied to a certain type of base dynamics and suitable sampling function.

This should be contrasted with our earlier results from \cite{dg1,dg2}: For every given base dynamics, the spectrum is purely singular continuous for a dense $G_\delta$ set of continuous sampling functions and it is purely absolutely continuous for a dense set of continuous sampling functions, with both statements holding uniformly in the choice of the initial point (i.e., element of the hull). Thus, the generic spectral type is singular continuous and from this perspective, the other spectral types must be rare. It is an open problem whether pure point spectrum occurs for a dense set of continuous sampling functions.

We would also like to emphasize that P\"oschel's general theorem applies to single Schr\"odinger operators and, whenever it applies, yields one such operator with uniformly localized eigenfunctions. In the context of Schr\"odinger operators with dynamically defined potentials, however, it is more natural to study the typical behavior of a member of the family of operators that results by varying the initial point. It is known that the spectral type is independent of it almost surely with respect to any ergodic measure. Limit-periodic (or, more generally, almost periodic) potentials in turn are uniquely ergodic, that is, there is a unique choice of such a measure -- the Haar measure on the hull. In our examples, we will even go beyond that and prove uniform localization results that hold uniformly for all elements of the hull. This is a novel phenomenon. Indeed, usually localization can be proved, and in fact holds, only almost surely. For random potentials, this is obvious since there are periodic realizations of the potential. For certain almost periodic potentials, there are results to this effect due to Jitomirskaya-Simon \cite{js} and Gordon \cite{g2}. Regarding results establishing pure point spectrum for all elements of the family, we are aware of the following: For the Maryland model, see \cite{fp,fg,gfp,pg,s}, which has an unbounded potential (and hence is not almost periodic), pure point spectrum was shown for the whole family but without uniform decay of eigenfunctions. There is some unpublished work of Jitomirskaya establishing a similar result for a bounded non-almost periodic model. To the best of our knowledge, in this paper we exhibit the first almost periodic example that is uniformly localized across the hull and the spectrum.

\section{Model and Result}

We consider Schr\"odinger operators $H_\omega$ acting on $\ell^2(\Z)$ with
dynamically defined potentials $V_\omega$ given by
\begin{equation}\label{equ:oper}
[H_\omega u](n) = u(n+1) + u(n-1) + V_\omega(n) u(n),
\end{equation}
where
\begin{equation}\label{equ:pot}
V_\omega (n) = f(T^n (\omega)) , \quad \omega \in \Omega, \; n \in
\Z
\end{equation}
with a homeomorphism $T$ of a compact space $\Omega$ and a
continuous sampling function $f : \Omega \to \R$.

\begin{defn}
We say that a family $\{u_k\} \subset \ell^2(\Z)$ is uniformly localized if there exist constants $r > 0$, called the decay rate, and $c < \infty$
such that for every element $u_k$ of the family, one can find $m_k \in \Z$, called the center of localization, so that $|u_k(n)| \leq c e^{-r|n-m_k|}$ for every $n \in \Z$. We say that the operator $H_\omega$ has ULE if it has a complete set of uniformly localized eigenfunctions.\footnote{Recall that a set of vectors is called complete if their span (i.e., the set of finite linear combinations of vectors from this set) is dense.}
\end{defn}

The notion of uniformly localized eigenfunctions and related ones were introduced by del Rio et al.\ in their comprehensive study of the question ``What is localization?'' \cite{dj2,dj}. As explained there, ULE implies uniform dynamical localization, that is, if $H_\omega$ has ULE, then
\begin{equation}\label{e.udl}
\sup_{t \in \R} \left| \left\langle \delta_n , e^{-itH_\omega} \delta_m \right\rangle \right| \le C_\omega e^{-r_\omega |n-m|}
\end{equation}
with suitable constants $C_\omega , r_\omega \in (0,\infty)$. While both properties are desirable, they are extremely rare. To quote from \cite{dj2}, ``the problem is that ULE does not occur'' and ``it is an open question, in fact, whether there is \emph{any} Schr\"odinger operator with ULE.'' Del Rio et al.\ may not have been aware of P\"oschel's work \cite{p} since it predates theirs and provides some examples of Schr\"odinger operators with ULE.

The occurrence of pure point spectrum for the operators $\{H_\omega\}_{\omega \in \Omega}$ is called \textit{phase stable} if it holds for every $\omega \in \Omega$. It is an unusual phenomenon since most known models are not phase stable. It is known that uniform localization of eigenfunctions (ULE) has a close connection with phase stability of pure point spectrum; compare the following theorem.

\begin{thm} \cite[Theorem C.1]{dj} \label{thm:ule}
If $H_\om$ has ULE for $\om$ in a set of positive $\mu$-measure, then $H_\om$ has pure point spectrum
for every $\om \in \mathrm{supp}(\mu)$, where $\mathrm{supp}(\mu)$ is the complement of the largest open set $S \subset \Omega$ for which $\mu(S) = 0$.
\end{thm}

In what follows, we will further assume that $\Omega$ is a Cantor
group that has a minimal translation $T$. Let us recall the
necessary definitions.

\begin{defn}
We say that $\Omega$ is a Cantor group if it is a totally
disconnected compact Abelian topological group with no isolated
points. A map $T : \Omega \to \Omega$ is called a translation if $T
(\omega) = \omega \cdot \om_0$ for some $\om_0 \in \Omega$, and
moreover, it is called minimal if the orbit $\{ T^n(\omega) : n \in
\Z \}$ of every $\omega \in \Omega$ is dense in $\Omega$.
\end{defn}

Jitomirskaya pointed out in \cite{j} that Theorem~\ref{thm:ule} can
be strengthened for a minimal $T$ in the sense that if there exists
some $\om_0$ such that $H_{\om_0}$ has ULE, then $H_\om$ has pure
point spectrum for every $\om \in \mathrm{supp}(\mu)$.

As explained by Gan in \cite{g}, Cantor groups that have minimal
translations are procyclic groups. We can classify such Cantor
groups by studying their frequency integer sets. Every Cantor group
with a minimal translation has a unique maximal frequency integer
set $S = \{n_k\} \subseteq \Z_+$ with the property that
$n_{k+1}/n_k$ is prime for every $k$. We will give more details
concerning this issue in a later section.

\begin{defn}
For a Cantor group that has a minimal translation, we say that it
satisfies the condition $\mathscr{A}$ if its maximal frequency
integer set $S = \{n_k\} \subseteq \Z_+$ has the following property:
there exists some integer $m \ge 2$ such that for every $k$, we have
$n_k < n_{k+1} \leq n^m_k,$ that is, $\log n_{k+1}/\log n_k$ is
uniformly bounded.
\end{defn}

We can now state our main result:

\begin{thm} \label{thm:main}
Suppose $\Omega$ a Cantor group that admits a minimal translation
$T$ and satisfies the condition $\mathscr{A}$. Then there exists
some $f \in C(\Omega,\R) $ such that for every $\om \in \Omega$, the
Schr\"odinger operator with potential $ f(T^n(\om))$ has ULE with $\omega$-independent constants. In particular, we have uniform dynamical localization \eqref{e.udl} for every $\omega$  with $\omega$-independent constants as well.
\end{thm}

We will heavily use P\"oschel's results in \cite{p}, which will be
recalled in Section~\ref{sec:poeschel}, to obtain the above theorem.
P\"oschel used an abstraction of KAM methods, with some of the basic
ideas going back to Craig \cite{cr}, R\"ussmann \cite{r} and Moser
\cite{m1}. In this approach, there is an important concept, that of
a distal sequence, which we will discuss in Section
\ref{sec:distal}. The first step in proving Theorem~\ref{thm:main}
is to construct a distal limit-periodic potential in our framework.
In Section \ref{sec:cantor}, we will recall the connection between
hulls of limit-periodic potentials and Cantor groups, which makes it
possible to embed our constructed distal limit-periodic potential
isometrically in $C(\Omega,\R)$.

\section{Limit-Periodic Potentials and Cantor Groups}
\label{sec:cantor}

This section addresses the connection between hulls of limit-periodic potentials and Cantor groups that have minimal translations, first introduced by Avila in \cite{a}, discussed to the extent needed in \cite{dg1,dg2}, and discussed in detail in \cite{g}. Since it will play an important role in this paper, we present some aspects of it here.

Let $\sig$ be the left shift operator on $\ell^\infty(\Z)$, that is, $(\sig(d))_n = d_{n+1}$ for every $d \in \ell^\infty(\Z).$ Let $\mathrm{orb}(d)
= \{ \sig^k (d) : k \in \Z\}$ and denote by $\mathrm{hull}(d)$ the closure of $\mathrm{orb}(d)$ in $\ell^\infty (\Z)$. Let us recall the following
standard definitions:

\begin{defn}
Consider a sequence $d \in \ell^\infty(\Z)$. It is called periodic if $\mathrm{orb}(d)$ is finite, it is called limit-periodic if it belongs to
the closure of the set of periodic sequences, and it is called almost periodic if $\mathrm{hull}(d)$ is compact.
\end{defn}

Every periodic sequence is limit-periodic and every limit-periodic sequence is almost periodic. For a limit-periodic $d \in \ell^{\infty}(\Z)$,  every $d^{'} \in \mathrm{hull}(d)$ is still limit-periodic. More precisely, we have the following result.

\begin{prop} \label{prop:avilahull} \cite[Lemma 2.1]{a}
Suppose $d$ is limit-periodic. Then, $\mathrm{hull}(d)$ is compact and has a unique topological group structure with identity $\sig^0(d) = d$ such that
$$
\phi : \Z \to \mathrm{hull}(d),\quad k \mapsto \sig^{k}(d)
$$
is a homomorphism. Also, the group structure is Abelian and there exist arbitrarily small compact open neighborhoods of $d$ in
$\mathrm{hull}(d)$ which are finite index subgroups.
\end{prop}

The last statement in the above proposition tells us that
$\mathrm{hull}(d)$ is totally disconnected. So if $d$ is not
periodic, $\mathrm{hull}(d)$ is a Cantor group. The translation $T$
defined initially on $\mathrm{orb}(d)$ by $T(\sigma^{i}(d)) =
\sigma^{i+1}(d)$ and extended to $\mathrm{hull}(d)$ by continuity is
minimal. There may be other minimal translations in
$\mathrm{hull}(d)$.

\begin{remark}
Note that not every Cantor group admits a minimal translation. For example, $$\Omega = \prod^{\infty}_{j=0} \Z_2,$$ where $\Z_2$ is a cyclic $2$-group, is a Cantor group with the product topology, but it has no minimal translations.
\end{remark}

\begin{prop} \label{prop:hull} \cite[Lemma 2.2.]{a}
Given a Cantor group $\Omega$, a minimal translation $T$, and $f \in C(\Omega,\R)$, define
$F : \Omega \rarrow \ell^\infty(\Z)$, $F(\om) = (f(T^n(\om)))_{n \in \Z}$.
Then we have that $F(\om)$ is limit-periodic and $F (\Omega)
= \mathrm{hull}(F(\om))$ for every  $\om \in \Omega$.
\end{prop}

The following lemma will play an important role below.

\begin{lemma} \label{lem:hulliso} \cite[Lemma 4.1]{g}
There exists some $f \in C(\Omega,\R)$ such that $\mathrm{hull}(F(e)) \cong \Omega$
{\rm (}where we denote, as above, $F(e) = (f(T^n(e)))_{n \in \Z}${\rm )}.
\end{lemma}

Moreover, we have

\begin{prop}  \label{prop:limitform} \cite[Corollary A.1.5]{as}
If $d \in \ell^{\infty}(\Z)$ is limit-periodic, then there exists a set $S_d = \{n_j\}_{j \ge 1} \subset \Z_+$ with $n_j | n_{j+1}$ for every $j$ such that
\begin{equation}\label{equ:limitform}
d(k) = \sum^{\infty}_ {j = 1} p_j
(k),
\end{equation}
with $n_j$-periodic $p_j \in \ell^{\infty}(\Z)$. This convergence is uniform.
\end{prop}

A set $S_d = \{n_j\}$ associated with $d$ as in this proposition will be called a \textit{frequency integer set} of $d$. Since one of the defining properties is that $n_j$ divides $n_{j+1}$ for every $j$, the elements of frequency integer sets are always listed in increasing order.

\begin{prop} \cite[Theorem 2.1]{g}\label{thm:classification}
Given limit-periodic potentials $d$ and $\tilde{d} \in \ell^{\infty}(\Z)$ with infinite frequency integer sets $S_d$ and
$S_{\tilde{d}}$ respectively, $\mathrm{hull}(d) \cong \mathrm{hull}(\tilde{d})$ if and only if for any $n_i \in S_d$ there
exists $m_j \in S_{\tilde{d}}$ such that $n_i | m_j$ and vice versa.
\end{prop}

Since the expansion \eqref{equ:limitform} is not unique, one may
have many frequency integer sets for $d$. A union of frequency
integer sets is still a frequency integer set of $d$. There exists a
unique maximal frequency integer set $M_d$ in the sense that every
frequency integer set $S_d$ is contained in $M_d$, and the maximal
frequency integer set is of the form $M_d = \{n_j\}$, where
$n_{j+1}/n_{j}$ are all primes. (We refer the reader to
\cite[Appendix 1]{as} and \cite[Section 2]{g} for more details about
the frequency integer sets of limit-periodic potentials.)

By Lemma \ref{lem:hulliso}, we know that for every Cantor group $\Omega$ that has minimal translations, there exists a limit-periodic
potential $d$ such that $\mathrm{hull}(d) \cong \Omega$. Thus, we can also endow such an $\Omega$ with a maximal frequency integer set $S_\Omega$.
Moreover, we have

\begin{lemma} \label{lem:subfrequency}\cite[Theorem 2.1]{g}
Given two Cantor groups $\Omega$ and $\tilde{\Omega}$ that have minimal translations, $\Omega \cong \tilde{\Omega}$ if and only if
they have the same maximal frequency integer set.
\end{lemma}

We need the following lemma.

\begin{lemma} \label{lem:iso}
Suppose we are given a Cantor group $\Omega$ and a minimal translation $T$. If $\mathrm{hull}(d) \cong \Omega$ with $d \in \ell^{\infty}(\Z)$, then
there is an $f \in C(\Omega,\R)$ such that $f(T^i(e)) = d_i$ for every $i \in \Z$.
\end{lemma}

\begin{proof}
By Lemma \ref{lem:hulliso} we have $\tilde{f} \in C(\Omega, R)$ such that $\mathrm{hull}((\tilde{f}(T^i(e)))_{i \in \Z}) \cong \Omega.$ Since
$\mathrm{hull}(d) \cong \Omega,$  we have a continuous isomorphism
$$h : \mathrm{hull}((\tilde{f}(T^i(e)))_{i \in \Z}) \to \mathrm{hull}(d)$$ with $h((\tilde{f}(T^i(e)))_{i \in \Z}) = d.$

Clearly, for $T^{n_k}(e) \in \Omega$ we have
$h((\tilde{f}(T^i(T^{n_k}(e))))_{i \in \Z}) = \sigma^{n_k}(d)$ since
$$(\tilde{f}(T^i(T^{n_k}(e))))_{i \in \Z} = \sigma^{n_k} ((\tilde{f}(T^i(e)))_{i \in \Z}).$$
If $\lim_{k \to \infty} T^{n_k}(e) = \omega,$ then $h((\tilde{f}(T^i(\om)))_{i \in \Z}) = \lim_{k\to \infty}\sigma^{n_k}(d)$, where
the limit exists since $h$ and $\tilde{f}$ are both continuous.
Define $f$ by $f(T^i(e)) = \sigma^i(d)_0 = d_i$.
We extend $f$ to the whole $\Omega$ by $f(\omega) = \lim_{k \to \infty} \sigma^{n_k}(d)_0$ if $\omega = \lim_{k \to \infty} T^{n_k}(e).$
By the previous analysis, $f$ is well defined and continuous. So there is an $f \in C(\Omega,\R)$ such that $(f(T^i(e)))_{i \in \Z} = d.$
\end{proof}

We see that, given a Cantor group $\Omega$ and a minimal translation
$T$, the elements of $C(\Omega,\R)$ parametrize a class of
limit-periodic potentials. Next, let us describe the periodic
elements of this class. Since $\Omega$ is Cantor, there exists a
decreasing sequence of Cantor subgroups $\Omega_k \subset \Omega$
with finite index $n_k$ such that $\bigcap \Omega_k = \{ e \}$.  We
say that $f \in C(\Omega,\R)$ is a periodic sampling function (of
period $n$) if $f(T^n(\om)) = f(\om)$ for every $\om \in \Omega$.
For $f \in C(\Omega,\R)$, we define
$$
f_{k}(\om) = \int_{\Omega_k}f(\om +
\tilde{\om})~d\mu_{\Omega_k}(\tilde\om),
$$
where $\mu_{\Omega_k}$ is the Haar measure on $\Omega_k$. Then
$f_{k}$ is an $n_k$-periodic sampling function. Clearly, there exist
compact subgroups with finite index contained in arbitrarily small
neighborhoods of $e$, and this shows that the set of periodic
sampling functions is dense in $C(\Omega, \R)$. Let $P_k$ be the set
of sampling functions which are defined on $\Omega / \Omega_k$. Then
$P_k \subset P_{k+1}$, the elements of $P_k$ are $n_k$-periodic, and
$P = \bigcup P_k$ is the set of all periodic sampling functions and
it is dense in $C(\Omega,\R)$.

\section{Distal Sequences}
\label{sec:distal}

In this section, we discuss approximation functions and distal sequences; compare \cite{p} and \cite{r}.

\begin{defn}
A function $Q(x) : [0,\infty)\to [1,\infty)$ is called an approximation function if both
$$
q(t) = t^{-4} \sup_{x \ge 0} Q(x) e^{-tx}
$$
and
\begin{equation} \label{equ:distalh}
h(t) = \inf_{\kappa_t} \prod^{\infty}_{i = 0} q(t_i)^{2^{-i-1}}
\end{equation}
are finite for every $t > 0$. In \eqref{equ:distalh}, $\kappa_t$ denotes the set of all sequences $t \ge t_1 \ge t_2 \ge \cdots \ge 0$ with $\sum{t_i} \leq t$.
\end{defn}

\begin{defn}
A sequence $d \in \ell^\infty(\Z)$ is called distal if for some approximation function $Q$, we have
$$
\inf_{i \in \Z} |d_i - d_{i+k}| \ge Q(|k|)^{-1}
$$
for every $k \in \Z \setminus \{0\}$.
\end{defn}

\begin{prop}
If $d \in \ell^{\infty}(\Z)$ is distal, then every $\tilde{d} \in \mathrm{hull}(d)$ is also distal.
\end{prop}

\begin{proof}
This follows readily from the definition.
\end{proof}

The following lemma shows how to generate distal sequences in our framework.

\begin{lemma}\label{lem:subdistal}
Given a Cantor group $\Omega$ satisfying the condition $\mathscr{A}$ and a minimal translation $T$, there exists an $f \in
C(\Omega, \R)$ such that $(f(T^i(e)))_{i \in \Z }$ is a distal sequence.
\end{lemma}

\begin{proof}
Given a Cantor group $\Omega$ and a minimal translation $T$, by Lemma~\ref{lem:hulliso} there is a limit-periodic potential $l$ such that $\mathrm{hull}(l) \cong \Omega$. Since $\Omega$ satisfies the condition $\mathscr{A}$, there exists $m \ge 2$ such that for the elements of its maximal frequency integer set $S_{\Omega} = \{ n_k \}$, we have $n_{k-1} < n_k \leq  n^m_{k-1}$ for every $k$.

Consider $S_\Omega$. Here we let $n_1 > 1.$ For $n_1 \in S_\Omega$,
there must exist some $n_k \in [n^3_1,n^{3m}_1]$. If not, we pick
the largest $n_i \in [n_1,n^3_1)$ and then $n_{i+1}$ will be
strictly larger than $n^{3m}_1$. Then we have $n_{i+1} > n^{3m}_1 >
n^m_i$ which contradicts the assumption. So we can pick $n_k$ such
that $n^3_1 \leq n_k \leq n^{3m}_1$. By induction, we can pick a
subset of $S_{\Omega}$ which we still denote by $I_0 = \{ n_k \}$
satisfying $n^3_k \leq n_{k+1} \leq n^{3m}_k$ for every $k \in
\Z^+$. Without any contradiction, we take $n_0 = 1$ for the following
computation.

Define $a_v (i)= j$ where $0 \leq j < n_v$ and $i=j\ (\mathrm{mod}\
n_v)$, so $a_v$ is $n_v$-periodic. Let $d = (d_i)_{i \in \Z}$ and
$d^{(k)} = (d^{(k)}_i)_{i \in \Z}$, where
$$
d_i = \sum^{\infty}_{v = 1} \frac{a_v(i)}{n^2_{v-1}n_v} \quad \text{ and } \quad d^{(k)}_i = \sum^{k}_{v = 1} \frac{a_v(i)}{n^2_{v-1} n_v}.
$$
By the divisibility property of any frequency integer set, $d^{(k)}$
is an $n_k$-periodic sequence. Since for every $i \in \Z$ and $k \in
\Z_+$ we have
$$
\left| d_i - d^{(k)}_i \right| = \left| \sum^{\infty}_{v= k+1} \frac{a_v(i)}{n^2_{v-1}n_v} \right| \leq \sum^{\infty}_{v= k+1} \frac{1}{n^2_{v-1}},
$$
it follows that $d^{(k)}$ converges to $d$ uniformly. Thus, $d$ is
limit-periodic and one of its frequency integer sets is $I_0$.

For any $i_1 \neq i_2$, fix $k$ so that $n_{k-1} \leq |i_{1} -
i_{2}| < n_k$. If $k = 1$, then $|d^{(1)}_{i_1} - d^{(1)}_{i_2}| \ge
\frac{1}{n_1}.$ Also, we have
\begin{align*}
|(d_{i_1} -d^{(1)}_{i_1}) - (d_{i_2} - d^{(1)}_{i_2}) | & \leq n_1
\sum^{\infty}_{v=2} \frac{1}{n^2_{v-1}n_v} \\
& \leq \frac{8}{7 n_1 n_2}\\
& \leq \frac{4}{7 n_1}
\end{align*}
So it is easy to see that $|d_{i_1} - d_{i_2}| \ge \frac{3}{7 n_1}
\ge \frac{2}{3 n^{3m+1}_1}.$

If $k \ge 2$, we have \be\label{equ:estidistal}
\frac{1}{n^2_{k-2}n_{k-1}}
> \frac{2(n_{k} - 1)}{n^2_{k-1} n_{k}}, \ee
since $n^3_i \leq n_{i+1} \leq n^{3m}_i$. Moreover, we have
\begin{align*}
|d^{(k)}_{i_1} - d^{(k)}_{i_2}| & = \left |\sum^{k}_{v = 1} \frac{(a_v(i_1) - a_v(i_2))}{n^2_{v-1}n_v}\right | \\
& = \left |\sum^{k-1}_{v = 1} \frac{(a_v(i_1) - a_v(i_2))}{n^2_{v-1}n_v} +  \frac{i_1 - i_2}{n^2_{k-1}n_k} \right |. \\
\end{align*}
From (\ref{equ:estidistal}) we conclude that $|\sum^{k-1}_{v = 1}
\frac{(a_v(i_1) - a_v(i_2))}{n^2_{v-1}n_v}|$ is $0$ or larger than
$\frac{2(n_{k} - 1)}{n^2_{k-1} n_{k}}.$ So $|d^{(k)}_{i_1} -
d^{(k)}_{i_2}| = \left |\sum^{k-1}_{v = 1} \frac{(a_v(i_1) -
a_v(i_2))}{n^2_{v-1}n_v} +  \frac{i_1 - i_2}{n^2_{k-1}n_k} \right |
\ge \frac{n_{k-1}}{n^2_{k-1}n_k} = \frac{1}{n_{k-1}n_k}.$
We also have
\begin{align*}
|(d_{i_1} - d^{(k)}_{i_1}) - (d_{i_2} - d^{(k)}_{i_2}) | & = \left| \sum^{\infty}_{v = k+1} \frac{(a_v(i_{1}) - a_v(i_{2}))}{n^2_{v-1}n_v} \right |  \\
& \leq n_k \sum^{\infty}_{v=k+1} \frac{1}{n^2_{v-1}n_v} \\
& \leq \sum^{\infty}_{v=0}\frac{1}{n_k n_{k+1} 4^v}  \\
& = \frac{4}{3 n_k n_{k+1}}.
\end{align*}
Thus, we get
\begin{align*}
|d_{i_1} - d_{i_2}| & \ge  \frac{1}{n_k n_{k-1}} - \frac{4}{3 n_k n_{k+1}} \\
& \ge \frac{2}{3 n_k n_{k-1}} \\
& \ge \frac{2}{3 n^{3m+1}_{k-1}}\\
& \ge \frac{2}{3 |i_{1}-i_{2}|^{3m+1}}.
\end{align*}
Therefore, $d$ is a distal sequence with an approximation function
$$
Q(x) = \begin{cases} \frac{3 n^{3m+1}_1}{2}, & 0 \leq x < n_1; \\
\frac{3x^{3m+1}}{2}, & x \ge n_1.
\end{cases}
$$

By Lemma~\ref{lem:subfrequency}, we have $\mathrm{hull}(d) \cong \mathrm{hull}(l) \cong \Omega$. By Lemma~\ref{lem:iso}
there is an $f \in C(\Omega, \R)$ such that $(f(T^i(e)))_{i \in \Z} = d$.
\end{proof}

\begin{remark}
For any $r \ge 0$, let
$$
G(x) = \begin{cases} 1, & 0 \leq x < 1; \\ x^r, & x > 1.
\end{cases}
$$
It is not hard to see that
$$
h(t) \leq c t^{-4-r}
$$
by choosing $t_i = t 2^{-i-1}$. The constant $c$ depends on $r$. It follows that $G(x)$ is an approximation function. In particular, this also shows that $Q(x)$ above is indeed an approximation function.
\end{remark}

\section{P\"oschel's Results} \label{sec:poeschel}

In this section we rewrite some of P\"oschel's results from \cite{p}, tailored to our purpose.

Let $\tau \ge 1$ be an integer and $\mathfrak{M}$ a Banach algebra
of real $\tau$-dimensional sequences $a = (a_i)_{i \in \Z^\tau}$
with the operations of pointwise addition and multiplication of
sequences. In particular, the constant sequence $1$ is supposed to
belong to $\mathfrak{M}$ and have norm one. Moreover, $\mathfrak{M}$
is required to be invariant under translation: if $a \in
\mathfrak{M}$, then $\| T_k a \|_{\mathfrak{M}} = \| a
\|_{\mathfrak{M}}$ for all $k \in \Z^\tau$, where $T_k a_i = a_{i
+ k}$.

We denote by $M$ the space of all matrices $A = (a_{i,j})_{i,j \in
\Z^\tau}$ satisfying $A_k = (a_{i, i+k}) \in \mathfrak{M}$, $k \in
\Z^\tau$, that is, $A_k$ is the $k$-th diagonal of $A$ and it is
required to belong to $\mathfrak{M}$. In $M$, we define a Banach
space
$$
M^s = \{A \in M, \p A\p_s < \infty\},\quad 0 \leq s \leq \infty,
$$
where
$$
\| A\|_s = \sup_{k \in \Z^\tau} \| A_k\|_{\mathfrak{M}} e^{|k|s}.
$$
Obviously,
$$
M^s \subset M^t,\quad \p \cdot \p_s \ge \p\cdot \p_t,\quad 0 \leq t \leq s \leq \infty.
$$
In particular, $M^{\infty}$ is the space of all diagonal matrices in $M$.

\begin{thm}[Theorem A, \cite{p}]\label{thm:poeschel1}
Let $D$ be a diagonal matrix whose diagonal $d$ is a distal sequence
for $\mathfrak{M}$. Let $0 < s \leq \infty$ and $0 < \sig \leq
\min\{1, \frac{s}{2} \}$. If $P \in M^s$ and $\p P \p_s \leq \delta
\cdot h(\frac{\sig}{2})^{-1}$, where $\delta > 0$ depends on the
dimension $\tau$ only, then there exists another diagonal matrix
$\tilde{D}$ and an invertible matrix $V$ such that
$$
V^{-1}(\tilde{D}+P)V = D .
$$
In fact, $V,V^{-1} \in M^{s-\sig}$ and $\tilde{D} - D \in M^{\infty}$ with $$\p V-I \p_{s-\sig}, \p V^{-1}-I \p_{s-\sig} \leq C \cdot \p P \p_s,$$
$$
\p \tilde{D} - D + [P]\p_{\infty} \leq C^2 \cdot \p P \p^2_s,
$$
where $C = \delta^{-1}\cdot h(\frac{\sig}{2})$, and $[\cdot]$
denotes the canonical projection $M^s \to M^{\infty}$. If $P$ is
Hermitian, then $V$ can be chosen to be unitary on
$\ell^{2}(\Z^\tau)$. Note that $h$ is the function
\eqref{equ:distalh} associated with $d$.
\end{thm}

An important consequence of the preceding theorem for discrete Schr\"odinger operators is the following.

\begin{thm}[Corollary A, \cite{p}]\label{thm:pcor}
Let $d$ be a distal sequence for some translation invariant Banach
algebra $\mathfrak{M}$ of $\tau$-dimensional real sequences. Then
for $0 \leq \eps \leq \eps_0, \eps_0 > 0$ sufficiently small, there
exists a sequence $\tilde{d}$ with $\tilde{d} -d \in \mathfrak{M}$,
$\p \tilde{d} - d \p_{\mathfrak{M}} \leq \frac{\eps^2}{\eps^2_0},$
such that the discrete Schr\"odinger operator
$$
(\tilde{H}u)_i = \eps \sum_{|l|=1}u_{i + l} + \tilde{d}_i u_i, \quad
i \in \Z^\tau
$$
has eigenvalues $\{d_i : i \in \Z^\tau \}$ and a complete set of
corresponding exponentially localized eigenvectors with decay rate
$1 + \log\frac{\eps_0}{\eps}$.
\end{thm}

Next let us discuss how to apply the above results.\\[2mm]
\noindent {\bf P\"oschel's Example.} Fix $\tau \ge 1$, and let
$\mathcal{P}$ be the set of all real $\tau$-dimensional sequences $a
= (a_i)$ with period $2^n, n \ge 0$, in each dimension; that is,
$a_i = a_j, i-j \in 2^n \Z^\tau$. The closure of $\mathcal{P}$ with
respect to the sup norm $\p\cdot \p_\infty$ is a Banach algebra,
which we denote by $\mathcal{L}$.  It is a subspace of the space of
all limit periodic sequences.

Let $\alpha_v,v\ge 1$, be the characteristic function of the set
$$
A_v =
\begin{cases} \bigcup_{N \in \Z} [N\cdot2^v, N\cdot2^v + 2^{v-1}), & v  \text{ even}; \\
\bigcup_{N \in \Z} [N\cdot2^v+2^{v-1},N\cdot2^v+2^v), & v  \text{ odd.}
\end{cases}
$$
Then, $\alpha_v$ has period $2^v$. Construct an $\tau$-dimensional
sequence $d = (d_i)$ such that
$$
d_i = \sum^{\infty}_{v = 1} \sum^\tau_{\mu=1}
\alpha_v(i_\mu)2^{-(v-1)\tau-\mu}, i=(i_1,\cdots,i_\tau) \in
\Z^\tau,
$$
belongs to $\mathcal{L}$ and lies dense in $[0,1]$. It is a distal sequence for $\mathcal{L}$ with
$$
\p(d-T_k d)^{-1}\p_\infty \leq 16^\tau |k|^\tau, \quad 0 \neq k \in
\Z^\tau.
$$

Applying Theorem~\ref{thm:pcor} to this distal sequence $d$, we find
that there exists $\tilde{d} \in \mathcal{L}$ and $\eps_0 > 0$ such
that for any $0 < \eps \leq \eps_0$, the discrete Schr\"odinger
operator with potential $(\frac{\tilde{d}_i}{\eps})_{i \in \Z}$ has
the pure point spectrum $\overline{\{\frac{d_i}{\eps}: i \in \Z^\tau  \}}$
and a complete set of exponentially localized eigenvectors with decay
rate $1+ \log{\frac{\eps_0}{\eps}}$. Moreover, the spectrum of this
Sch\"odinger operator as a set is $\overline{\{\frac{d_i}{\eps}: i
\in \Z^\tau \}} = [0, \frac{1}{\eps}]$ since $\overline{\{d_i: i \in
\Z^\tau  \}} = [0,1]$.

\section{Proof of Theorem~\ref{thm:main}}
We are now ready to give the proof of Theorem \ref{thm:main}. Given
a Cantor group $\Omega$ that admits a minimal translation
$T$ and satisfies the condition $\mathscr{A}$, we fix a
metric $\p \cdot \p$ compatible with the topology. We have already
seen that there exists some $f \in C(\Omega, \R)$ such that
$d=(f(T^{i}(e)))_{i \in \Z}$ is a distal sequence; compare Lemma~\ref{lem:subdistal}. Clearly,
$C(\Omega,\R)$ will induce a class of limit-periodic potentials. We
denote it by $\mathcal{B}$, and one can check that this class is a
translation invariant Banach algebra with the $\ell^\infty$-norm. By
Theorem~\ref{thm:pcor}, there exists a sufficiently small $\eps_0 > 0$
such that for $0 < \eps \leq \eps_0$, there is a sequence $\tilde{d}
\in \mathcal{B}$ with $\p\tilde{d}-d \p_{\infty} \leq
\frac{\eps^2_0}{\eps^2}$ so that the discrete Schr\"odinger operator
$$
(Hu)_i = u_{i-1} + u_{i+1} + \frac{\tilde{d_i}}{\eps}u_i,\quad i \in \Z
$$
has eigenvalues $\{\frac{d_i}{\eps}, i \in \Z\}$ and a complete set of corresponding exponentially localized eigenvectors with decay
rate $r = 1+ \log{\frac{\eps_0}{\eps}}$. There exists a sampling function
$\tilde{f} \in C(\Omega, \R)$ such that $\tilde{f}(T^i(e)) = \frac{\tilde{d_i}}{\eps}$ since $\tilde{d} \in \mathcal{B}.$

For the Schr\"odinger operator $H$ associated with potential
$\tilde{f}(T^{i}(e))$, denote its matrix representation with respect
to the standard orthonormal basis of $\ell^2(\Z)$, $\{ \delta_n
\}_{n \in \Z}$, by the same symbol. P\"oschel's theorem also implies
that there exists a unitary $V: \ell^2(\Z) \to \ell^2(\Z)$ (with
corresponding matrix denoted by the same symbol) such that
\begin{equation}\label{equ:jacob1}
H \cdot V = V \cdot D,
\end{equation}
where $D$ is a diagonal matrix with the diagonal $D_{0} =
(\frac{d_i}{\eps})_{i \in \Z}$. We write $V = (\cdots, V_{-1}, V_0,
V_1,\cdots)$ where $V_i$ is the $i$-th diagonal of $V$, and
similarly, we write $H = (\cdots,0,0,H_{-1},H_0,H_1,0,0,\cdots)$ and
$D = (\cdots,0,0,D_{0},0,0,\cdots)$.  Moreover, by
Theorem~\ref{thm:poeschel1} we have that $V \in M^{r}$, where $r >
0$ and $M^r$ is a space of matrices associated with the Banach
algebra $\mathcal{B}$ (see Section~\ref{sec:poeschel} for the
description of this space). (Note that $V \in M^{r}$ follows from
\cite[Proof of Corollary~A]{p}.) Since $V \in M^r$, we have $\p
V\p_r = \sup_{i \in \Z} \p V_i \p_{\infty} e^{|i|r} < C$ where $C$
is a constant. So $\p V_i \p_\infty < C e^{-r|i|}, \forall i \in
\Z$. Let $V^{(j)}$ be the $j$-th column of $V$, that is, $V^{(j)}$
is an eigenfunction of $H$. Since $V^{(j)}(k) = V^{(k+(j-k))}(k)$,
$V^{(j)}(k)$ is also an entry in $V_{j-k}$, and so $|V^{(j)}(k)| <
Ce^{-r|j-k|}$. $C$ is independent of $j$, so the corresponding
Schr\"odinger operator $H$ has ULE. This property is strong enough
to imply that the pure point spectrum of $H$ is independent of $\om$
\cite{j}, that is, it is phase stable. In order to see this more
explicitly, we would like to prove it in our framework, and
furthermore, show that for other $\om$, the associated Schr\"odinger
operator still has ULE with the same constant $C$. Note that the
latter property does not follow from Theorem \ref{thm:ule}.

We have the following lemma.

\begin{lemma}
Suppose we are given matrices $A,B \in \R^{\Z \times \Z}$, one of which has only finitely many non-zero diagonals. Then, we have for the $k$-th diagonal of $Z = AB$,
$$
Z_k = \sum_{l \in \Z} A_l \cdot T^{l}(B_{k-l}),
$$
where $\cdot$ is the pointwise multiplication {\rm (}i.e., $A_l \cdot
T^{l}(B_{k-l})$ is still a sequence{\rm )} and $T$ is the translation defined by
$(T(B_{k-l}))_i = (B_{k-l})_{i+1}$ for $i \in \Z$.
\end{lemma}

\begin{proof}
Since for $i,k \in \Z$, we have
\begin{align*}
z_{i,i+k} & = \sum_{t \in \Z} a_{i,t}b_{t,i+k} \\
&= \sum_{l \in \Z} a_{i,i+l}b_{i+l,i+k} \\
&= \sum_{l \in \Z}a_{i,i+l}b_{i+l,i+l+k-l},
\end{align*}
the lemma follows.
\end{proof}

Now consider a given $\om \in \Omega$. By Proposition~\ref{prop:hull} we have $(\tilde{f}(T^{i}(\om)))_{i \in \Z} \in \mathrm{hull}((\tilde{f}(T^{i}(e)))_{i \in \Z})$. If $\omega$ is in the orbit of $e$, that is, $\om = T^{t}(e)$ for some $t \in \Z$, ULE with the same constants and eigenvalues follows from unitary operator equivalence directly. However, we write this out in detail so that we see clearly what happens in the case where $\om$ can only be approximated by elements of the form $T^{t}(e)$.

By the previous lemma, \eqref{equ:jacob1} is equivalent to the following form:
$$
\forall k\in \Z : \qquad \sum_{l \in \Z} H_l \cdot T^l (V_{k-l}) = \sum_{l \in \Z} V_l \cdot T^l(D_{k-l}).
$$
Since $D_j = 0$ for $j \not= 0$ and $H_{\pm 1}$ are both constant equal to one, this simplifies as follows,
$$
\forall k\in \Z : \qquad T^{-1} V_{k+1} + H_0 \cdot V_{k} + T V_{k-1} = V_k \cdot T^k(D_0).
$$

If the potential is replaced by $\tilde{f}(T^{i+t} (e))$, with the matrix
$\tilde{H} = (\cdots,0,0,\tilde{H}_{-1},\tilde{H}_{0},\tilde{H}_{1},0,0,\cdots)$
such that $\tilde{H}_j(i)=H_{j}(i+t), j \in \{-1,0,1\}$, we still have
$$
\forall k\in \Z : \qquad T^{-1} \tilde V_{k+1} + \tilde H_0 \cdot \tilde V_{k} + T \tilde V_{k-1} = \tilde V_k \cdot T^k(\tilde D_0),
$$
where $\tilde{V}_k(i) = V_k(i+t),\ k \in \Z $ and $\tilde{D}_0(i) = D_0(i+t)$. Reversing the steps above, this means that
$$
\tilde H \cdot \tilde V = \tilde V \cdot \tilde D.
$$
We can conclude that $\tilde{H}$ has the pure point spectrum $\overline{\{\frac{d_{i+t}}{\eps} : i \in \Z \}} = \overline{\{\frac{d_i}{\eps}:i\in \Z \}}$. Moreover, $\tilde{V} = (\cdots, \tilde{V}_{-1}, \tilde{V}_{0}, \tilde{V}_{1}, \cdots)$ is the eigenfunction matrix of $\tilde{H}$, and for any $i,k \in \Z$, $|\tilde{V}_k (i)| = |V_k (i+t)| \leq Ce^{-r|k|}$. So for the eigenfunction $\tilde{V}^{(j)}$ of $\tilde{H}$, we still have $|\tilde{V}^{(j)}(i)| < Ce^{-r|j-i|}$, and hence ULE with the same constants follows.

If $\lim_{m \to \infty} T^{t_m}(e) = \om$, that is, $\tilde{f}(T^{i}(\om)) = \lim_{m \to \infty}\tilde{f}(T^{i+t_m}(e))$, then for $\tilde{f}(T^{i+t_m}(e))$, we have already seen that
\begin{equation}\label{equ:shift}
\tilde{H}^{(m)}\cdot \tilde{V}^{(m)} = \tilde{V}^{(m)} \cdot \tilde{D}^{(m)}.
\end{equation}
Let $\tilde{V}^{(m)}_k$ be the $k$-th diagonal of $\tilde{V}^{(m)}$, so that $\tilde{V}^{(m)}_k(i) = V_k(i+t_m)$. There exists some $\tilde{f}_k \in C(\Omega, \R)$ such that $\tilde{V}^{(m)}_k(i) = V_{k}(i+t_m) = \tilde{f}_k(T^{i + t_m}(e))$. So $\lim_{m \to \infty} \tilde{V}^{(m)}_k(i)  = \lim_{m \to \infty} \tilde{f}_k (T^{i+t_m}(e)) = \tilde{f}_k (T^{i}(\om))$, and we denote $\tilde{f}_k (T^{i}(\om))$ by $\tilde{V}^{(\infty)}_k (i)$. Similarly, $\lim_{m \to \infty} \tilde{D}^{(m)}$ exists and $\tilde{D}^{(\infty)}_0(i) = f(T^i(\om))$, where $\tilde{D}^{(\infty)}_0$ is the 0-th diagonal of $\tilde{D}^{(\infty)}$. Thus, as we let $m \to \infty$, \eqref{equ:shift} takes the following form:
\begin{equation}\label{equ:limit}
\tilde{H}^{(\infty)} \cdot \tilde{V}^{(\infty)} = \tilde{V}^{(\infty)} \cdot \tilde{D}^{(\infty)},
\end{equation}
where $\tilde{H}^{(\infty)}$ is (the matrix representation of) the Schr\"odinger operator with potential $\tilde{f}(T^i(\om))$. Equation~\eqref{equ:limit} implies that $\tilde{H}^{(\infty)}$ has
the pure point spectrum  $\overline{\{\frac{d_i}{\eps} : i \in \Z\}}$, and its eigenfunctions are uniformly localized since $|(\tilde{V}^{(\infty)})^{(j)}(k)| < Ce^{-r|j-k|}$ for any $j, k \in \Z$, where $(\tilde{V}^{(\infty)})^{(j)}$ is the $j$-th column of $\tilde{V}^{(\infty)}$. This completes the proof of Theorem~\ref{thm:main}. \hfill \qedsymbol

\section{Open Problems}

We conclude this paper with a number of open problems concerning the spectral properties of limit-periodic Schr\"odinger operators that we regard as interesting.

Given the results of \cite{dg1,dg2}, it would be desirable to complete the topological picture. Thus, given a minimal translation $T$ of a Cantor group $\Omega$, consider for $f \in C(\Omega,\R)$ and $\omega \in \Omega$ the spectral type of the associated Schr\"odinger operator $H_\omega$ with potential given by $V_\omega(n) = f(T^n (\omega))$.

\medskip

\textit{Problem 1.} Is it true that for $f$ from a suitable dense subset of $C(\Omega,\R)$, $H_\omega$ has pure point spectrum for (Haar-) almost every $\omega \in \Omega$?

\medskip

We already know that for generic $f \in C(\Omega,\R)$, $H_\omega$ has purely singular continuous spectrum for every $\omega \in \Omega$, and also that for $f$ from a suitable dense subset of $C(\Omega,\R)$, $H_\omega$ has purely absolutely continuous spectrum for every $\omega \in \Omega$. Thus, an affirmative answer to Problem~1 would clarify the effect of the choice of $f$ on the spectral type. Since the methods of P\"oschel are essentially restricted to large potentials, one should not expect them to yield an answer to Problem~1 and one should in fact pursue methods involving some randomness aspect.

Note, however, the different quantifier on $\omega$ in Problem~1, compared to the results just quoted. In this paper, we exhibit $(\Omega,T,f)$ for which $H_\omega$ has pure point spectrum for every $\omega \in \Omega$. From this perspective, the following problem arises naturally:

\medskip

\textit{Problem 2.} Is the spectral type of $H_\omega$ always the same for every $\omega \in \Omega$?

\medskip

For quasi-periodic potentials, this is known not to be the case (cf.~\cite{js}). However, the mutual approximation by translates for two given elements in the hull is stronger in the limit-periodic case than in the quasi-periodic case, so it is not clear if similar counterexamples to uniform spectral types exist in the limit-periodic world.

Another related problem is the following:

\medskip

\textit{Problem 3.} Is the spectral type of $H_\omega$ always pure?

\medskip

Again, in the quasi-periodic world, this is known not to be the case: there are examples that have both absolutely continuous spectrum and point spectrum (cf.~\cite{b1,b2}).

Returning to the issue of point spectrum, one interesting aspect of the result stated (in the continuum case) by Molchanov and Chulaevsky in \cite{mc} is the coexistence of pure point spectrum with the absence of non-uniform hyperbolicity. That is, in their examples, the Lyapunov exponent vanishes on the spectrum and yet the spectral measures are pure point. This is the only known example of this kind and it would therefore be of interest to have a complete published proof of a result exhibiting this phenomenon. Especially since our study is carried out in a different framework, we ask within this framework the following question:

\medskip

\textit{Problem 4.} For how many $f \in C(\Omega,\R)$ does the Lyapunov exponent vanish throughout the spectrum and yet $H_\omega$ has pure point spectrum for (almost) every $\omega \in \Omega$?

\medskip

Given the existing ideas, it is conceivable that Problems~1 and 4 are closely related and may be answered by the same construction. If this is the case, it will then still be of interest to show for a dense set of $f$'s that there is almost sure pure point spectrum with \emph{positive} Lyapunov exponents.

\end{document}